\documentclass[12pt]{article}
\usepackage{caption}
\usepackage{url}
\usepackage{color}
\usepackage{amssymb}
\usepackage{amsmath}
\usepackage{amsthm}
\usepackage{graphicx}
\usepackage{mathscinet}

\newtheorem{theorem}{Theorem}[section]
\newtheorem{lemma}[theorem]{Lemma}

\newtheorem{cor}[theorem]{Corollary}

\DeclareMathOperator{\mad}{mad}

\title{Coloring the square of a sparse graph $G$ with almost $\Delta(G)$ colors}
\author{Matthew P. Yancey \thanks{Institute for Defense Analyses / Center for Computing Sciences (IDA / CCS), mpyance@super.org}}

\begin{document}
\maketitle

\begin{abstract}
For a graph $G$, let $G^2$ be the graph with the same vertex set as $G$ and $xy \in E(G^2)$ when $x \neq y$ and $d_G(x,y) \leq 2$.
Bonamy, L\'ev\^{e}que, and Pinlou conjectured that if  $\mad (G) < 4 - \frac{2}{c+1}$ and $\Delta(G)$ is large, then $\chi _\ell(G^2) \leq \Delta(G) + c$.
We prove that if $c \geq 3$, $\mad (G) < 4 - \frac{4}{c+1}$, and $\Delta(G)$ is large, then $\chi _\ell(G^2) \leq \Delta(G) + c$.
Dvo\v{r}\'ak, Kr\'{a}\soft{l}, Nejedl\'{y}, and \v{S}krekovski conjectured that $\chi(G^2) \leq \Delta(G) +2$ when $\Delta(G)$ is large and $G$ is planar with girth at least $5$; our result implies $\chi (G^2) \leq \Delta(G) +6$.
\end{abstract}

\section{Motivation}

For a fixed graph $G$, let $G^2$ be the graph such that $V(G^2) = V(G)$ and $E(G^2) = E(G) \cup \{uw: u \neq w, N(u) \cap N(w) \neq \emptyset\}$.
A \emph{$2$-distance coloring} of $G$ is a proper coloring of $G^2$; a \emph{$2$-distance list coloring} is a list coloring of $G^2$.
Let $\chi^2(G) = \chi(G^2)$ and $\chi _\ell^2(G) = \chi _\ell(G)$.
The study of these chromatic numbers has been spurred by three major conjectures.
Wegner \cite{W} conjectured that when $\Delta(G) \geq 8$, then $\chi^2(G) \leq \lceil 1.5\Delta(G) \rceil + 1$.
The introduction to \cite{CS} contains a survey on progress towards solving this conjecture.
Kostochka and Woodall \cite{KW} conjectured that $\chi _{\ell}^2(G) = \chi^2(G)$; recently this has been proven to not be true \cite{KP}.
The girth of a graph $G$ is the length of the shortest cycle in $G$ and is denoted $g(G)$.
Wang and Lih conjectured that there exists a function $M$ such that if $G$ is a planar graph with $g(G) \geq 5$ and $\Delta(G) \geq M(g(G))$, then $\chi^2(G) = \Delta(G)+1$.
Wang and Lih's conjecture is true only on the restricted domain $g(G) \geq 7$ \cite{BGINT,BIN}; if $g(G) = 6$ then the weaker result $\chi^2(G) \leq \Delta(G) +2$ \cite{DKNS} is true.
Dvo\v{r}\'ak, Kr\'{a}\soft{l}, Nejedl\'{y}, and \v{S}krekovski conjectured that $\chi^2(G) \leq \Delta(G) +2$ when $\Delta(G) > M$, $G$ is planar, and $g(G) = 5$ \cite{DKNS}.

Motivated by Wang and Li's conjecture, there have been a series of results showing that graphs with bounded maximum average degree have $\chi^2(G)$ close to $\Delta(G)$.
Let $n(G) = |V(G)|$ and $e(G) = |E(G)|$; the \emph{maximum average degree} of $G$, denoted by $\mad(G)$, is the maximum of $2\frac{e(H)}{n(H)}$, taken over all non-empty subgraphs $H\leq G$.
The family of planar graphs with girth $g$ is a subfamily of graphs with maximum average degree less than $\frac{2g}{g-2}$.
Dolama and Sopena \cite{DS} proved that if $\Delta(G) \geq 4$ and $\mad(G) < 16/7$, then $\chi^2(G) = \Delta(G) + 1$.
Bonamy, L\'ev\^{e}que, and Pinlou \cite{BLP} and independently Cranston and \v{S}krekovski \cite{CS} proved that there exists a function $M$ such that if $\mad(G) < 2.8 - \epsilon$ and $\Delta(G) > M(\epsilon)$, then $\chi _\ell^2(G) = \Delta(G) + 1$.
This was later improved by Bonamy, L\'ev\^{e}que, and Pinlou \cite{BLP2} that if $\mad(G) < 3 - \epsilon$ and $\Delta(G) > M(\epsilon)$, then $\chi _\ell^2(G) = \Delta(G) + 1$.
This is sharp: if we only assumed $\mad(G) < 3$, then this would imply Wang and Li's conjecture is true for girth $6$ graphs, which is a contradiction.

Bonamy, L\'ev\^{e}que, and Pinlou \cite{BLP} proved that if $\mad(G) \leq 4 - \frac{40}{c+16}$, then $\chi _\ell^2(G) \leq \Delta(G) +c$.
Charpentier (see \cite{BLP2}) gave a construction of a graph with $\mad(G) < 4 - \frac{2}{c+1}$ and $\chi _\ell^2(G) = \Delta(G) +c + 1$.
Bonamy, L\'ev\^{e}que, and Pinlou conjectured that Charpentier's construction is optimal \cite{BLP2} - that is they ask if it is true that $\mad(G) < 4 - \frac{2}{c+1} - \epsilon$ and $\Delta(G) > M(\epsilon)$ implies that $\chi _\ell^2(G) \leq \Delta(G) +c$?
Their result states that this conjecture is true when $c=1$; we provide the strongest progress yet for when $c \geq 3$.

\begin{theorem} \label{2main}
Let $c$ be a fixed number such that $c \geq 3$.
If $G$ is a graph such that $\mad (G) < 4 - \frac{4}{c+1} - \epsilon$ for some $\frac{4}{c(c+1)} > \epsilon > 0$, then 
$\chi ^2_\ell(G) \leq \max\{ \Delta(G) + c, 16 c^2 \epsilon^{-2} \}$.
\end{theorem}

If one could omit the $\epsilon$ term, then the case $c=2$ of Bonamy, L\'ev\^{e}que, and Pinlou's conjecture would imply Dvo\v{r}\'ak, Kr\'{a}\soft{l}, Nejedl\'{y}, and \v{S}krekovski's conjecture.
Charpentier's construction is not planar when $c \geq 2$ (although the length of the shortest cycle is $5$).
Our result is strong enough to provide a partial result towards Dvo\v{r}\'ak, Kr\'{a}\soft{l}, Nejedl\'{y}, and \v{S}krekovski's conjecture.
Wang and Lih \cite{WL} proved that if $G$ is a planar graph with $g(G) \geq 5$, then $\chi^2_\ell(G) \leq \Delta(G) + 16$.

\begin{cor}
If $G$ is a planar graph with $g(G) \geq 5$ and $\Delta(G) \geq 63500$, then $\chi^2_\ell(G) \leq \Delta(G) + 6$.
\end{cor}

\section{Proof of Theorem \ref{2main}}
Let $c,\epsilon$ be as stated in Theorem \ref{2main}.
Let $K_c(G) = \max\{ \Delta(G) + c, 16 c^2 \epsilon^{-2} \}$.
Because $\epsilon < \frac{4}{c(c+1)}$, we have that $K_c(G) \geq 4c^2$.
We will show that $\chi_\ell^2(G) \leq K_c(G)$.

We use the notation ``$x$ is in conflict with $y$'' to say that $xy \in E(G^2)$.
We call a vertex \emph{massive} if it has degree at least $\sqrt{K_c(G)}$.
A vertex is \emph{type one} if it is massive, or has degree at least $3$ while being adjacent in $G$ to a massive vertex.
A vertex is \emph{type two} if it has degree at least three and it is not type one.
We will frequently use the fact that if $w$ is type two, then the number of vertices in conflict with $w$ is less than $K_c(G)$.

For a graph $G$, let $n_1(G)$ and $n_2(G)$ denote the number of vertices of type one and type two vertices in $G$, respectively.
We define $L(G) = (n_1(G) + n_2(G), e(G))$, and we order graphs lexicographically by $L(G)$.
That is, we say that $H$ is smaller than $G$ if 
\begin{enumerate}
	\item $n_1(H) + n_2(H) < n_1(G) + n_2(G)$, or
	\item $n_1(H) + n_2(H) =n_1(G) + n_2(G)$ and $e(H) < e(G)$.
\end{enumerate}
It is easy to see that if $H$ is a subgraph of $G$, then $H$ is smaller than $G$.
Furthermore, if $H$ is a subgraph of $G$, then $K_c(H) \leq K_c(G)$.
By way of contradiction, assume that $G$ is a counterexample to the theorem that is minimal by our ordering.
Because $c \geq 1$, it follows that $\delta(G) \geq 2$.

\begin{lemma} \label{2 by 2}
If $xy \in E(G)$, then $\max\{d(x), d(y)\} \geq 3$.
\end{lemma}
\begin{proof}
Suppose $d(x) = d(y) = 2$.
Let $G' = G - x - y$, and so by induction there exists a $2$-distance coloring of $G'$ using at most $K_c(G') \leq K_c(G)$ colors.
Because $c \geq 3$, each of $x$ and $y$ are in conflict with less than $K_c(G)$ vertices, and so our coloring of $G'$ can be extended greedily into a $2$-distance coloring of $G$.
This contradicts the fact that $G$ is not $2$-distance colorable.
\end{proof}

\begin{lemma} \label{small is type one}
If $3 \leq d(u) \leq c-1$, then $u$ is type one.
\end{lemma}
\begin{proof}
By way of contradiction, assume that $d(u) \leq c-1$ and $u$ is type two. 
Let $Y \subseteq N(u)$ be the set of neighbors of $u$ that are type one, and let $y = |Y|$.
Let $G'$ be $G$ with $u$ removed and replaced with ${y \choose 2}$ vertices $x_i$, such that the neighborhoods of the $x_i$ form the subsets of $Y$ of order $2$.
It is clear that $n_1(G') \leq n_1(G)$ and $n_2(G') \leq n_2(G) - 1$, and so $G'$ is smaller than $G$ in our ordering.
By induction, there exists a $2$-distance coloring of $G'$ using at most $K_c(G')$ colors.

We claim that $K_c(G') \leq K_c(G)$.
There are two possibilities that we must account for: if (1) $\Delta(G') + c > K_c(G)$ or (2) $\mad(G') > \mad(G)$.
The only vertices with a larger degree in $G'$ than in $G$ are those in $Y$.
By assumption, each $v \in Y \subseteq N(u)$ satisfies $d_G(v) \leq \sqrt{K_c(G)}$.
By construction, $d_{G'}(v) - d_G(v) = y - 2 \leq d(u) < c$, and so $d_{G'}(v) \leq \sqrt{K_c(G)} + c < K_c(G)$.
This proves that (1) may not happen; now we concern ourselves with (2).

For $S \subseteq V(G)$, let $\rho_G(S) = (2 - 2(c+1)^{-1} - \epsilon/2)|S| - e(G[S])$.
Note that $\mad(G) < 4 - 4(c+1)^{-1} - \epsilon$ is equivalent to $\rho_G(S) > 0$ for all $\emptyset \neq S \subseteq V(G)$.
By way of contradiction, suppose that there exists a non-empty $S'$ such that $\rho_{G'}(S') \leq 0$.
Without loss of generality, assume that $S'$ minimizes $\rho_{G'}(S')$ among all such sets; this implies that $\delta(G'[S']) \geq 2$.
Moreover, $x_i \in S'$ if and only if $N(x_i) \subset S'$.
Let $X' = N(u) \cap S'$, let $s = |X'|$, and let $S = S' \cap V(G)$ if $s \leq 1$ and $S = S' \cap V(G) + u$ otherwise.
If $s \leq 2$, then $G[S] \cong G'[S']$ and therefore $\rho_{G'}(S') = \rho_G(S)$.
If $s \geq 3$, then
\begin{small}
\begin{eqnarray*}
\rho_{G'}(S') - \rho_G(S) 	&=&	\left( (2-2(c+1)^{-1} - \epsilon/2) {s \choose 2} - 2 {s \choose 2}\right) - ((2 - 2(c+1)^{-1} - \epsilon/2)-s) \\
				& = &	\left( \frac{2}{c+1} + \frac{\epsilon}2\right)\left(- {s \choose 2} + 1\right) + s - 2 \\
				&\geq&	(s-2)\left(1-\left(\frac{2}{c+1} + \frac{\epsilon}2\right)  \frac{s+1}2 \right)
\end{eqnarray*}
\end{small}
Because $\epsilon < \frac{4}{c(c+1)}$ and $s \leq d(u) \leq c-1$, we see that $0 \geq \rho_{G'}(S') \geq \rho_G(S)$ for any value of $s$.
This is a contradiction, and so (2) never happens.
Therefore our claim that $K_c(G') \leq K_c(G)$ is true.

So we have a $2$-distance coloring on $G'$ using at most $K_c(G)$ colors.
By construction, no conflicting pairs of vertices in $G - (N[u] - Y)$ share a color.
Every vertex in $N[u] - Y$ is either type two or has degree $2$ while being adjacent to a vertex with degree less than $c$; therefore every vertex in $N[u] - Y$ has less than $K_c(G)$ conflicts. 
It follows that they can be greedily recolored so that they do not share a color with any vertex they are in conflict with.
This contradicts that $G$ is not $2$-distance colorable with $K_c(G)$ colors.
\end{proof}

\begin{lemma} \label{either big or both type ones}
If $N(u) = \{x,y\}$ and $x$ is not type one, then $y$ is massive.
\end{lemma}
\begin{proof}
By way of contradiction, assume that $x$ is not type one and $d(y) \leq \sqrt{K_c(G)}$.
By Lemma \ref{2 by 2}, $d(x) \geq 3$, and so $x$ is type two.
Let $G' = G - u$, and so by induction there exists a $2$-distance coloring of $G'$ using at most $K_c(G') \leq K_c(G)$ colors.
We will extend this coloring in two steps to a $2$-distance coloring of $G$, which is a contradiction.
First, we re-color $x$ so that it does not share the same color with $y$ (and this is possible because $x$ is type two).
Second, we color $u$, which has at most $2\sqrt{K_c(G)} < K_c(G)$ conflicts.
\end{proof}

\begin{lemma} \label{small meets small}
Let $N(u) = \{x,y,z_3, z_4, \ldots, z_m\}$, where $d(z_i) = 2$ for each $i$ and $3 \leq m = d(u) \leq c-1$.  
Under these conditions, $m + d(x) \geq c+2$.
\end{lemma}
\begin{proof}
By way of contradiction, assume that $d(x) + m \leq c + 1$.
Let $N(z_i) = \{u,v_i\}$ for each $i$.
Let $G' = G - z_3$, and so by induction there exists a $2$-distance coloring of $G'$ using at most $K_c(G') \leq K_c(G)$ colors.
We will extend this coloring in two steps to a $2$-distance coloring of $G$, which is a contradiction.
Our first step is to color $u$ so that it does not have the same color as any vertex (besides itself) in $N[x] \cup N[y] \cup (v_i)_{3 \leq i \leq m}$.
There are at most $\Delta(G) + d(x) + m - 2$ conflicts in that set, and by assumption this is less than $K_c(G)$.
The second step is to recolor each $z_i$ that has the same color as $u$.
Each $z_i$ has at most $\Delta(G) + d(u) < K_c(G)$ conflicts, and so this is possible.
\end{proof}

We are now prepared to describe the discharging procedures.
Each vertex begins with charge equal to its degree.
In the end, we will show that each vertex has final charge at least $4 - \frac{4}{c+1} - \epsilon$, which contradicts that $\mad(G) < 4 - \frac{4}{c+1} - \epsilon$.

\begin{enumerate}
	\item If $u$ is type one and $d(u) < c$, then $u$ gives $1-\frac{2}{c+1}$ charge to each neighbor with degree $2$.
	\item If $\frac{c+2}2 \leq d(u) \leq c-1$, then $u$ gives $\frac{2d(u)}{c+1} - 1$ charge to each neighbor whose degree is at least $3$ and less than $c$.
	\item If $u$ is type one and $c \leq d(u) \leq \sqrt{K_c(G)}$, then $u$ gives $1-\frac{2}{c+1}$ charge to each neighbor that is not massive.
	\item If $u$ is type two, then $u$ gives $1 - \frac{4}{c+1} + \frac{\epsilon}c$ charge to each neighbor.
	\item If $u$ is massive, then $u$ gives $1 - \epsilon/c$ charge to each neighbor.  
\end{enumerate}

We now calculate a lower bound on the final charge for each vertex.
\begin{enumerate}
	\item If $d(u) = 2$, then by Lemma \ref{2 by 2} both neighbors of $u$ have degree at least $3$.  If both are type one, then the final charge on $u$ is at least $2 + 2(1 - \frac{2}{c+1})$.  If one of them is type two, then by Lemma \ref{either big or both type ones} the other neighbor is massive and therefore the final charge on $u$ is $2 + (1 - \frac{4}{c+1} + \frac{\epsilon}c) + (1 - \epsilon/c)$.
	\item Suppose $3 \leq d(u) \leq c-1$ and $u$ is adjacent to at least $c-2$ neighbors with degree $2$.  By Lemma \ref{small is type one} one of the neighbors is massive, by Lemma \ref{small meets small} a second neighbor has degree at least $c+2-d(u)$, and so $u$ has exactly $c-2$ neighbors with degree $2$.  Because $\frac{2(c+1-d(u))}{c+1} - 1 = -(\frac{2d(u)}{c+1} - 1)$, we have that the net transfer of charge between $u$ and that second neighbor is that $u$ ``receives'' at least $1 - \frac{2d(u)}{c+1}$ charge, regardless of whether $d(u) \geq \frac{c+2}2$ or $d(u) < \frac{c+2}2$.  So the final charge on $u$ is at least $d(u) + (1-\epsilon/c) + ( 1 - \frac{2d(u)}{c+1}) - (d(u) - 2)(1 - \frac{2}{c+1}) =  4 - \frac{4}{c+1} - \epsilon/c$.
	\item Suppose $3 \leq d(u) \leq c-1$ and $u$ is adjacent to at most $c-3$ neighbors with degree $2$. By Lemma \ref{small is type one} one of the neighbors of $u$ is massive.  
	\begin{itemize}
		\item If $d(u) \geq \frac{c+1}2$, then the final charge on $u$ is at least $d(u) + (1 - \epsilon/c) - (d(u)-3)(1 - \frac{2}{c+1}) - 2(\frac{2d(u)}{c+1}-1) = 6 - \frac{2(3 + d(u))}{c+1} - \epsilon/c$.  Because $d(u) \leq c-1$, this is greater than $ 4 - \frac{4}{c+1} - \epsilon$.
		\item If $d(u) < \frac{c+1}2$, then the final charge on $u$ is at least $d(u) + (1 - \epsilon/c) - (d(u)-3)(1 - \frac{2}{c+1}) = 4 - \frac{2(3 - d(u))}{c+1} - \epsilon/c$.  Because $d(u) \geq 3$, this is greater than $ 4 - \frac{4}{c+1} - \epsilon$.
	\end{itemize}
	\item If $c \leq d(u) \leq \sqrt{K_c(G)}$ and $u$ is type one, then the final charge on $u$ is at least $d(u) + (1-\epsilon/c) - (d(u) - 1)(1-\frac{2}{c+1}) = 2 + \frac{2d(u)-2}{c+1} - \epsilon/c$.
	\item If $u$ is type two, then Lemma \ref{small is type one} states that $d(u) \geq c$.  So the final charge on $u$ is at least $d(u)(\frac{4}{c+1} - \epsilon/c) \geq 4 - \frac{4}{c+1} - \epsilon$.
	\item If $u$ is massive, then the final charge on $u$ is at least $\sqrt{K_c(G)}\epsilon c^{-1} \geq 4$.
\end{enumerate}



\begin{thebibliography}{1}

	\bibitem{BLP}	M. Bonamy, B. L\'ev\^{e}que, and A. Pinlou, ``$2$-distance Coloring of Sparse Graphs,'' \textit{Journal of Graph Theory} \textbf{77} 3 (2014), 190 -- 218.

	\bibitem{BLP2}	M. Bonamy, B. L\'ev\^{e}que, and A. Pinlou, ``List coloring the square of sparse graphs with large degree,'' \url{http://arxiv.org/abs/1308.4197} .

	\bibitem{BGINT} O. Borodin, A. Glebov, A. Ivanova, T. Neustroeva, and V. Tashkinov, ``Sufficient Conditions for the $2$-distance $(\Delta+1)$-colorability of plane graphs.'' \textit{Siberian Electronic Mathematical Reports} \textbf{1} (2004), 129--141.

	\bibitem{BIN} O. Borodin, A. Ivanova, and T. Neustroeva, ``$2$-distance coloring of sparse plane graphs.'' \textit{Siberian Electronic Mathematical Reports} \textbf{1} (2004), 76 -- 90.

	\bibitem{CS} D. Cranston and R. \v{S}krekovski, ``Sufficient sparseness conditions for $G^2$ to be $(\Delta+1)$-choosable, when $\Delta \geq 5$.'' \textit{Discrete Applied Mathematics} \textbf{162} (2014), 167--176.

	\bibitem{DS} M. Dolama and `'E. Sopena, ``On the maximum average degree and the incidence chromatic number of a graph,'' \textit{Discrete Math. Theor. Comut. Sci.} \textbf{7} (2005), 203 -- 216.

	\bibitem{DKNS} Z. Dvo\v{r}\'ak, D. Kr\'{a}\soft{l}, P. Nejedl\'{y}, and R. \v{S}krekovski, ``Coloring squares of planar graphs with girth six,'' \textit{European Journal of Combinatorics} \textbf{29} (2008), 838--849.

	\bibitem{KP}  S. Kim and B. Park, ``Counterexamples to the list square coloring conjecture,'' \url{http://arxiv.org/abs/1305.2566} .

	\bibitem{KW} A. Kostochka and D. Woodall, ``Choosability conjectures and multicircuits,'' \textit{Discrete Math.} \textbf{240}, 1–3 (2001), 123--143.

	\bibitem{WL} F. Wang and K Lih, ``Labeling planar graphs with conditions on girth and distance two,'' \textit{SIAM Journal of Disc. Math.} \textbf{17}, 2 (2003), 264 -- 275.

	\bibitem{W} G. Wegner, ``Graphs with given diameter and a coloring problem,'' Technical Report, University of Dortmund, 1977.

\end{thebibliography}
\end{document}